\documentclass{article}
\usepackage{graphicx} 
\usepackage{amssymb}
\usepackage{amsthm}
\usepackage{amsmath}
\usepackage{enumitem}

\providecommand{\dfrac}{\displaystyle \frac}

\providecommand{\N}{\mathbb{N}}
\providecommand{\Z}{\mathbb{Z}}
\providecommand{\bars}{\overline}

\providecommand{\Syl}{\textnormal{Syl}}
\providecommand{\Aut}{\textnormal{Aut}}
\providecommand{\Irr}{\textnormal{Irr}}
\providecommand{\Inn}{\textnormal{Inn}}
\providecommand{\cd}{\textnormal{cd}}

\newtheorem{introthm}{Theorem}
\newtheorem{thm}{Theorem}[section]
\newtheorem{cor}[thm]{Corollary}
\newtheorem{lem}[thm]{Lemma}

\newtheorem*{fct}{Fact}

\title{Groups Having a Character of Maximal Degree}

\author{Sara Jensen\footnote{sjensen1@carthage.edu; Department of Mathematics; Carthage College; 2001 Alford Park Dr; Kenosha, WI 53140; USA.} \and Mark L. Lewis\footnote{lewis@math.kent.edu; Department of Mathematical Sciences; P.O. Box 5190; 1300 Lefton Esplanade; Kent, OH 44242; USA.}}

\date{\today}

\begin{document}
\maketitle
\begin{abstract}
Let $G$ be a group, let $d$ be a character degree, and let $e$ be the integer so that $|G| = d(d+e)$.  It has been shown when $e > 1$ that $|G| \le e^4 - e^3$.  In this paper, we consider the groups where $|G| = e^4 - e^3$.  It is known that $e$ must be a power of a prime.  We classify the groups where $e$ is a prime and where $e$ is $4$, $9$, and $25$.  In so doing, we find a new nonsolvable Camina pair.
\end{abstract}

\textbf{2020 Mathematics Subject Classification: 20C15} 

\section{Introduction}

Throughout this paper, all groups are finite.  Let $G$ be a group and let $d$ be the degree of an irreducible character of $G$.  Usually, we take $d$ to be the degree of the largest irreducible character of $G$.  We know that $|G| > d^2$, so there is a positive integer $e$ so that $|G| = d(d+e)$.  

A group $G$ satisfies either $|G| = 2$ or is a $2$-transitive Frobenius group if and only if $e = 1$ (see Theorem 7 of \cite{berk}).  Since we can obtain a two transitive Frobenius group by taking the full multiplicative group of a finite field act on its additive group, there is no bound on $|G|$ when $G$ is a two-transitive Frobenius group, and thus, we cannot bound $|G|$ when $e = 1$.  

In \cite{snyder}, Snyder proves when $e > 1$ that $|G|$ is bounded by a function of $e$.  In particular, Snyder shows that $|G| \le (2(e!))^2$.  Isaacs greatly improves this bound in \cite{issnyd} by showing there is a universal constant $B$ so that $|G| \le B e^6$ and gives examples of groups $G$ where $|G| = e^4 - e^3$.  Durfee and the first author show in \cite{durjen} that in fact $|G| \le e^6 + e^4$ when $e$ is a prime power and $|G| < e^4 + e^3$ when $e$ is divisible by two distinct primes.  In addition, when $G$ has a nontrivial normal abelian subgroup and $e$ is divisible by two distinct primes, then $|G| < e^4 - e^3$.  The second author shows in \cite{lewissnyder} that when $G$ has a nontrivial abelian normal subgroup that $|G| \le e^4 - e^3$ and in general, that $|G| < e^4 + e^3$ and with Hung and Schaeffer Fry in \cite{hung2016finite} show that when $G$ has a nonabelian minimal normal subgroup, then $|G| < e^4 - e^3$, and so, for all values of $e > 1$ it is the case that $|G| \le e^4 - e^3$.   Furthermore, in \cite{hung2016finite}, they ask about the structure of the groups $G$ where $|G| = e^4 - e^3$.

For the remainder of this paper, we will assume that $e > 1$.  We will say that $G$ is an {\it Isaacs group} of degree $e$ if $|G| = e^4 - e^3$ where $e$ is an integer so that $|G| = d (d+e)$ and $d$ is a character degree of $G$ with $d = e^2 - e$.  

Our goal is to begin to understand the structure of Isaacs groups.  As mentioned above, in Theorem B (b) \cite{durjen}, Durfee and the first author prove that if $e$ is divisible by two primes and $G$ has a nontrivial abelian normal subgroup, then $|G| < e^4 - e^3$.  On the other hand, Hung, the second author, and Schaeffer-Fry show in Theorem 1.2 of \cite{hung2016finite} that if $G$ has a nonabelian minimal normal subgroup, then $|G| < e^4 - e^3$.  It follows that all minimal normal subgroups of an Isaacs group must be abelian and $e$ is a power of a prime.   

In \cite{hung2016finite}, it is shown an Isaacs group of degree $e  = p^a$ where $p$ is a prime and integer $a$ is one of the groups Gagola studied in \cite{gagola}.  We will summarize the known results regarding Isaacs groups and the groups Gagola studied in Section \ref{background}.  For our purposes here, it is enough to know that such a group will have a unique minimal normal subgroup and that unique minimal normal subgroup will have order $e$ where $e$ is the degree for $G$ as an Isaacs group and $N$ is an elementary abelian $p$-group where $p$ is the prime dividing $e$.  In our first result, we able to determine information about $O_p (G)$.

\begin{introthm}\label{intro 1}
Let $p$ be a prime, let $a \ge 1$ be an integer, let $G$ be an Isaacs group of order $p^a$, let $N$ be the unique minimal normal subgroup of $G$, and let $K = O_p (G)$.  Then $p^a < |K:N| \le p^{2a}$, $N \le Z (K)$, $K$ is not abelian, and $|K:K'| \ge p^a$.
\end{introthm}

Next, we show that we have reasonable control of the structure of $G$ when $G$ is an Isaacs group of degree $p^a$ that is $p$-closed.  We will review the definition, history, and background of two transitive Frobenius groups in Section \ref{background}.

\begin{introthm}\label{intro 2}
Let $p$ be a prime and $a$ be a positive integer.  If $G$ is an Isaacs group of degree $p^a$ that is $p$-closed, then $G$ has a normal Sylow $p$-subgroup $P$ so that $|Z(P)| = p^a$ and $|P:Z (P)| = p^{2a}$ and a Hall $p$-complement $H$ of order $p^a-1$ so that $Z(P) H$ is a two transitive Frobenius group.
\end{introthm}

Using this theorem, we will see that any nonsolvable Isaacs group where $G$ is $p$-closed will occur when $a = 2$, and we will show in Section \ref{p squared} that it seems unlikely that any examples will occur.

Unfortunately, when $G$ is not $p$-closed, the results we can obtain in general are much weaker.  We can show that when $p$ is odd, our group is close to being split over $O_p (G)$.

\begin{introthm}\label{intro 3}
Let $p$ an odd prime and let $a \ge 1$ be an integer.  Suppose that $G$ is an Isaacs group of degree $p^a$, suppose $S$ is a Sylow $2$-subgroup of $G$, take $Z$ be the unique subgroup of order $2$ in $S$, and write $K = O_p (G)$.   Then $G = K N_G (Z)$ and $N_G (Z) \cap K = C_K (Z) > 1$.
\end{introthm}

We begin by focusing on some ``small'' cases.  In particular, we classify all of the Isaacs groups whose degree is $p$ where $p$ is a prime.  We will see that the case $p = 2$ is trivial to dispense with, so we handle the case when $p$ is odd.

\begin{introthm}\label{intro 4}
Let $p$ be an odd prime, and let $G$ be an Isaacs group of degree $p$.  Then $G$ is $p$-closed, the normal Sylow $p$-subgroup is extra-special of order $p^3$, and a Hall $p$-complement is cyclic of order $p-1$.  For each prime $p$, there is (up to isomorphism) one Isaacs group of degree $p$ whose Sylow $p$-subgroup has exponent $p^2$ and there are (up to isomorphism) $(p-1)/2$ Isaacs groups of degree $p$ whose Sylow $p$-subgroup has exponent $p$.
\end{introthm}

We also consider the Isaacs groups of degree $p^2$.  We will see that when $G$ is a $p$-closed Isaacs group of degree $p^2$, then the normal Sylow $p$-subgroup is either semi-extraspecial or has nilpotence class $3$ and the structure is very limited.  When $G$ is not $p$-closed, we see that $O_p (G)$ has order $p^5$ and again we see that there are two possible structures for $O_p (G)$, one of which has nilpotence class $2$ and the other has nilpotence class $3$.  We see that the structure of $G/O_p (G)$ will be prescribed.  In particular, we will see in Theorem \ref{p^2 not pclosed} that any examples for $p \ge 5$ will necessarily be nonsolvable.

We determine all of the Isaacs groups for degrees $4$, $9$, and $25$.  We will see that for $e = 4$ there are four Isaacs groups that are $2$-closed and all of them have a normal Sylow $2$-subgroup that is semi-extraspecial, and there are two Isaacs groups that are not $2$-closed both of which have that $O_2 (G)$ has nilpotence class $2$.  One these has previously appeared in \cite{gagola} and one of which is new.   We will see that there are thirteen Isaacs groups of degree $9$ that are $3$-closed and one that is not $3$-closed.  The one that is not $3$-closed has previously appeared in \cite{gagola}.  Finally, there are twenty-four Isaacs groups of degree $25$ that are $5$-closed and two that not $5$-closed.  We close by noting that these two that are not $5$-closed are not solvable (or even $5$-solvable) answering (negatively) a question raised in \cite{hung2016finite} about whether Isaacs groups must be solvable.  One of the examples has that $O_5 (G)$ has nilpotence class $3$ and this example has previously appeared in \cite{nonsolv}.  The second example has that $O_5 (G)$ has nilpotence class $2$ and is new in this paper.  This example is important since Isaacs groups are examples of Camina pairs, and this is the third example of a nonsolvable Camina pair after the example in \cite{nonsolv} and the ones in Theorem 1.20 of \cite{GGLMNT}.

\section{Background}\label{background}

In this section, we provide some background on the results that we need from the literature.  First, we mention two transitive Frobenius groups.  A nice description of two-transitive Frobenius groups can be found in Section 5 of \cite{dmn}.  Two-transitive Frobenius groups can be obtained using a near-field.   There is an infinite family of near-fields, often called Dickson near-fields, and seven exceptional near-fields.  The Dickson near-fields and four of the exceptional near-fields give two-transitive Frobenius groups that are solvable.  There are three exceptional near fields that give two-transitive Frobenius groups that are not solvable.  These three near fields have orders $11^2$, $29^2$, and $59^2$.  In particular, there exist three nonsolvable two-transitive Frobenius groups: $(Z_{11} \times Z_{11})\rtimes {\rm SL} (2,5)$, $(Z_{29} \times Z_{29}) \rtimes ({\rm SL} (2,5) \times Z_7)$, and $(Z_{59} \times Z_{59}) \rtimes ({\rm SL} (2,5) \times Z_{29})$.  Another perspective on two-transitive Frobenius groups can be found in Sections 18 and 19 of \cite{pass}.

As we state in the introduction, Isaacs groups are examples of groups that were first studied by Gagola in \cite{gagola}.  In that paper, Gagola considers a group $G$ that has an irreducible character that vanishes on all but two conjugacy classes of $G$.  We will say that $\chi \in \Irr (G)$ is a {\it Gagola character} if $\chi$ vanishes on all but two conjugacy classes of $G$, and we will say that $G$ is a {\it Gagola group} if $G$ has a Gagola character.   Assuming that $|G| \ne 2$, Gagola shows that $G$ must have a unique minimal normal subgroup $N$ and that $N$ will be an elementary abelian $p$-group for some prime $p$.

Under the additional assumption that $G$ is not a $p$-group, Gagola shows that $N < O_p (G)$ and if $\lambda \in \Irr (N)$, then there is a Sylow $p$-subgroup $P$ of $G$ so that $P$ is the stabilizer of $\lambda$ in $G$ and $\lambda$ is fully-ramified with-respect to $P/N$.  If $|N| = p^a$, then there is an integer $b$ so that $|P:N| = p^{2b}$, and we have $|G:P| = p^a -1$ and $\chi (1) = p^b(p^a - 1)$.  In addition, if $H$ is a $p'$-subgroup of $G$, then $HN$ is a Frobenius group.

Let $N$ be a normal subgroup of $G$.  We will write $\Irr (G \mid N)$ for the set of irreducible characters of $G$ that do not have $N$ in the kernel.

Gagola groups are a special case of a broader class of groups we will need to consider.  Let $G$ be a group and let $N$ be a normal subgroup.  We say that $(G,N)$ is a {\it Camina pair} if every element $x \in G \setminus N$ is conjugate to all of the elements in the coset $xN$.  There are many equivalent conditions for this definition.  One of these is that all of the characters in $\Irr (G \mid N)$ vanish on $G \setminus N$.  When $G$ is a Gagola group with Gagola character $\chi$ and $N$ is the unique minimal normal subgroup, it is not difficult to show that $\Irr (G \mid N) = \{ \chi \}$ and that $\chi$ vanishes on $G \setminus N$, so $(G,N)$ will be a Camina pair.  The second author has written an expository paper that covers many of the known results for Camina pairs (see \cite{lewis2018camina}).

It is known that if $(G,N)$ is a Camina pair, $N$ is a $p$-group, and $P$ is a normal Sylow $p$-subgroup of $G$, then $(P,N)$ is a Camina pair.  Conversely, if $(P,N)$ is a Camina pair and $H$ is a $p'$-group that acts on $P$ so that $HN$ is a Frobenius group, then taking $G = PH$ we have that $(G,N)$ is a Camina pair.  Unfortunately, no such nice description has been found for Camina pairs where $N$ is a $p$-subgroup and a Sylow $p$-subgroup is not normal. 

Another related concept that we need is of a semi-extraspecial group.  Let $p$ be a prime.  We say a $p$-group $G$ is semi-extraspecial if every maximal subgroup $N$ in $Z(G)$ satisfies that $G/N$ is an extraspecial group.   It appears that semi-extraspecial groups were first studied by Beisiegel \cite{beis}.  One can show that $G' = Z(G)$ and every irreducible character of $G$ is fully-ramified with respect to $G/G'$, so there is a positive integer $n$ so that $|G:G'| = p^{2n}$ and $\cd (G) = \{ 1, p^n \}$.  This shows that $(G,G')$ is a Camina pair.  In fact, any Camina pair of the form $(G,G')$ where $G$ is nilpotent of nilpotence class $2$, then $G$ will be a semi-extraspecial $p$-group for some prime $p$.  The second author also has a paper where many of the known properties of semi-extraspecial groups are discussed (see Sections 2 - 7 of \cite{ses-expos}).

If $G$ is semi-extraspecial with $|G:G'| = p^{2n}$ and $|G'| = p^m$, then Beisiegel showed that $m \le n$ (see Satz 1 of \cite{beis}.  The semi-extraspecial groups where $m = n$ seem to be especially useful so they are called {\it ultraspecial}.  

Now, Camina pairs of the form $(G,G')$ are also known as Camina groups, and it is known that all Camina groups of nilpotence class $2$ are Camina pairs (see Theorem 1.2 of \cite{ver}).  It is not difficult to see that the extraspecial groups are precisely the semi-extraspecial groups whose centers have order $p$.   For each prime $p$ and each positive integer $m$, there are up to isomorphism two extraspecial groups of order $p^{2m +1}$.  When $p$ is odd, one has exponent $p$ and one has exponent $p^2$.    

When classifying $p$-groups, it is often useful to apply isoclinisms and this is certainly true for semi-extraspecial groups.  We say that groups $G$ and $H$ are {\it isoclinic} if there exist isomorphisms $\alpha: G/Z(G) \rightarrow H/Z(H)$ and $\beta: G' \rightarrow H'$ so that $[\alpha (g_1Z(G)),\alpha(g_2Z(G))] = \beta ([g_1,g_2])$.  One word of explanation is needed for the last equation. all choices for elements in the cosets $\alpha (g_i Z(G))$ yield the same answer, so that equation is actually well-defined. It is not difficult to see that if $G$ and $H$ are isomorphic, then they are isoclinic.  However, the extraspecial groups of the same order are always isoclinic, so being isoclinic is weaker than isomorphism.  It is not difficult to show that isoclinism is an equivalence relation.  Interestingly, it does not need to preserve order.  However, if $G$ and $H$ are isoclinic and have the same order, then $G$ is semi-extraspecial if and only if $H$ is semi-extraspecial.  When $p$ is odd, it is known that every semi-extraspecial group is isoclinic to a unique group of exponent $p$.

Now, Verardi proves in Corollary 5.11 of \cite{ver} for every prime $p$ that there is a  unique isoclinism class of ultraspecial groups of order $p^6$.

\section{General results}

We know from Theorem 1.2 of \cite{hung2016finite} and Theorem B (b) of \cite{durjen} that for $G$ to be an Isaacs group of degree $e$ that $e$ must be a power of prime, and so, for the remainder of the paper, we will be assuming that $e = p^a$ for some prime $p$ and an integer $a \ge 2$.  In this section, we are going to include some general results.  Then we obtain some results for general primes $p$ when $e = p^2$, and we follow up with a description of the situation for $e = 4$, $9$, and $25$.  

Recall that we are considering an Isaacs group $G$ where $|G| = e^4 - e^3$ with $d = e^2 -e$.   Applying Theorem 7.2 of \cite{hung2016finite}, we see that $G$ must be a Gagola group whose unique minimal normal subgroup has order $e$.  Using Lemma 2.2 of \cite{lewissnyder}, we know that $|P:N| = e^2$ and $d = e (|N| - 1)$.  Since $d = e(e-1)$, this implies that $|N| = e$.  By Corollary 2.3 of \cite{gagola}, we know that $d = e |G:P|$, so $|G:P| = e - 1$ where $P$ is a Sylow $p$-subgroup of $G$.   In light of Chillag and Mann \cite{chma} and amplified by our recent paper \cite{mypre}, we see that Gagola groups may be viewed as Camina pairs and the Camina pairs where the Camina kernels are $p$-groups breaks into two cases depending on whether or not $G$ is $p$-closed.  

In fact, we will see that when $G$ is $p$-closed, we can give an algorithm for constructing all groups that satisfy our condition.  In the non $p$-closed situation, we will only describe the groups that occur when $e = p^2$.  At this time, we do not have enough understanding to describe what occurs when $e = p^a$ for $a \ge 3$ and $G$ is not $p$-closed.

\begin{thm} \label{pclosed}
Let $p$ be a prime and $a \ge 1$ is an integer.  Then $G$ is an Isaacs group of degree $e = p^a$ that is $p$-closed if and only if $G$ has a normal Sylow $p$-subgroup $P$ so that $(P,Z(P))$ is a Camina pair with $|Z(P)| = e$ and $|P:Z (P)| = e^2$, and where $G$ has a Hall $p$-complement $H$ of order $e-1$ so that $Z(P) H$ is a two transitive Frobenius group.   \end{thm}

\begin{proof}
We are assuming that $e = p^a$ where $p$ is a prime and $a \ge 2$ is an integer.  First, we suppose that $G$ is an Isaacs group that is $p$-closed, so that $|G| = e^4 - e^3 = e^3 (e-1)$.   Let $P$ be a Sylow $p$-subgroup of $G$.  We may use the Schur-Zassenhaus theorem to see that $G$ has a Hall $p$-complement $H$.  We saw above that $|G:P| = e - 1$, so $|H| = e-1$.  Note that $|P| = e^3$ and since $G$ is $p$-closed, we see that $P$ is normal in $G$,  As mentioned above, we know that $G$ has a unique minimal normal subgroup $N$ and that $|N| = e$.  It follows that $N \le P$.   In Theorem 2.5 \cite{gagola}, it is proved that $N = Z (P)$.  This implies that $|Z(P)| = e$ and $|P:Z(P)| = e^2$.  Since $d = e^2 - e = e(e-1)$, we see that $e \in \cd(P)$.  This implies that $(P,Z(P))$ is a Camina pair.  It is implied (but not explicitly proved) in \cite{gagola} that $Z(P) H$ is a two transitive Frobenius group.

Conversely, suppose $G$ has a normal Sylow $p$-subgroup $P$ so that $(P,Z(P))$ is a Camina pair so that $|Z(P)| = e$ and $|P:Z (P)| = e^2$ and $G$ has a Hall $p$-complement $H$ of order $e-1$ so that $Z(P) H$ is a two transitive Frobenius group.  We see that $|P| = |P:Z(P)||Z(P)| = e^2 \cdot e = e^3$ and $|G| = |P||H| = e^3 (e-1) = e^4 - e^3$.  Obviously, $G$ is $p$-closed.  Let $\lambda$ be a nonprincipal irreducible character of $Z(P)$.  Since $(P,Z(P))$ is a Camina pair, we know that $\lambda$ is fully-ramified with respect to $P/Z(P)$.  That is, $\lambda^P$ has a unique irreducible constituent $\theta$.  We see that $\lambda$ is invariant in $P$, and since $Z(P) H$ is a Frobenius group, we have $C_H (\lambda) = 1$.  Since $\lambda$ is fully-ramified, we see that $C_H (\theta) = 1$, and so, the stabilizer of $\theta$ in $G$ is $P$.  This yields $q(q-1) \in \cd(G)$, and the theorem is proved.
\end{proof}

Note in the notation of Theorem \ref{pclosed}, we take $N = Z (P)$ and we know that $G$ is a Gagola group so that $N$ is the unique minimal normal subgroup or alternatively, we can say that $(G,N)$ is a Camina pair.

Using Theorem \ref{pclosed}, we have an algorithm for finding all of the $p$-closed groups that satisfy Snyder's equality with $|G| = e^4 - e^3$.  

(1) Find all Camina pairs $(P, Z(P))$ so that $|P| = e^3$ and $|Z(P)| = e$.  Notice that this equivalent to finding all groups $P$ of order $e^3$ so $Z(P)$ is elementary abelian of order $e$ and $e \in \cd (P)$.

(2) For groups $P$ satisfying (1) find all subgroups $H$ of ${\rm Aut} (P)$ having order $e-1$ so that $H$ acts 2-transitively and Frobeniusly on $Z(P)$.

(3) For groups $P$ satisfying (1) and $H$ satisfying (2), take $G = P \rtimes H$.

For $e = 4$, we do not need this algorithm.  We can actually search the small group library for all examples.  We will used this algorithm to describe all of the $p$-closed examples when $e = 9$ and when $e = 25$.

We will also describe how to find nonsolvable examples arising from nonsolvable two transitive Frobenius groups. In particular, we see that the possibilities for nonsolvable Isaacs groups that are $p$-closed occur only when $a = 2$ and $p$ is one of $11$, $29$, or $59$.  We will discuss these possibilities further when we specialize to $e = p^2$. 

\begin{cor}\label{nonsolvable p-closed}
Let $p$ be a prime and let $a \ge 1$ be an integer.   Suppose that $G$ is an Isaacs group that is $p$-closed with notation as in Theorem \ref{pclosed}.  Then $G$ is nonsolvable if and only if $Z(P)H$ is a nonsolvable two transitive Frobenius group.  In this case, $e = p^2$ and one of the following occurs: (1) $p = 11$ and $H \cong {\rm SL}_2 (5)$, (2) $p = 29$ and $H \cong {\rm SL}_2 (5) \times Z_{11}$, or (3) $p = 59$ and $H \cong {\rm SL}_2 (5) \times Z_{29}$.
\end{cor}

\begin{proof}
Since $P$ is solvable, $G$ is nonsolvable if and only if $H$ is nonsolvable.
\end{proof}

In the rest of this section, we are going to obtain information regarding $O_p (G)$ when $G$ is an Isacs group of degree $p^a$ that applies whether or not $G$ is $p$-closed.  We are going to obtain this information by looking at the Sylow subgroups of $G$ for primes not equal to $p$ that act on $p$.  There are going to be two interesting cases: the Zsigmondy prime case and the case of $2$ when $p$ is odd.  We will begin by describing the Zsigmondy prime case.

Let $m$ and $a$ be positive integers.  We say the prime $p$ is a {\it Zsigmondy prime divisor} of $m^a - 1$ if $p$ divides $m^a - 1$ and $p$ does not divide $m^b -1$ for $1 \le b < a$.  In the Zsigmondy prime theorem, Zsigmondy proved that there exist Zsigmondy primes for all pair $(m,a)$ except when $m$ is a Mersenne prime and $a = 2$ or $m = 2$ and $a = 6$.  In this next lemma we see that Zsigmondy primes have a useful property in our situation, and we can find a substitute when they do not exist.  However, we consider groups in a more general setting.

We will say that $p$ is a {\it generalized Zsigmondy prime divisor} of $m^a - 1$ if $p$ is a prime and there is a positive integer $n$ so that $p^n$ divides $m^a - 1$ and $p^n$ does not divide $m^b - 1$ for any integer $b$ with $1 \le b < a$.  It is easy to see that Zsigmondy prime divisors will be generalized Zsigmondy prime divisors.  Also, when $m$ is a Mersenne prime and $a = 2$, it is not difficult to see that $2$ will be a generalized Zsigmondy prime divisor for $m^2 - 1$.  A quick check shows shows that $p = 3$ (with $n =2$ ) will be a generalized Zsigmondy prime divisor when $m =2$ and $a = 6$.  Thus, all pairs $(m,a)$ have a generalized Zsigmondy prime divisor.  We now obtain a lemma involving generalized Zsigmondy prime divisors.  Notice that we have fairly general hypotheses on this lemma.

\begin{lem} \label{prime divisors}
Let $p$ be a prime and let $a \ge 1$ be an integer.  Let $G$ be a group with $K = O_p (G)$ with a subgroup $N \le Z(K)$ such that $|N| = p^a$ and $N < K$.  Suppose for every subgroup $H \le G$ with $p$ not dividing $|H|$ that the subgroup $NH$ is a Frobenius group and that $p^a-1$ divides $|G|$.  Take $q$ to be a generalized Zsigmondy prime divisor of $p^a - 1$, take $Q$ to be a Sylow $q$-subgroup of $G$, and $Z$ to be the unique subgroup of order $q$ in $Q$.  Then the following are true:
\begin{enumerate}
\item 
If $K' \le X < Y \le K$ are $Q$-invariant subgroups so that $Y/X$ that is irreducible under the action of $Z$, then either $Q$ centralizes $Y/X$ or $Q$ acts faithfully on $Y/X$ and $|Y/X| = p^a$.  
\item $p^a$ divides $|K:K'|$ and $Q$ acts nontrivially on $K/K'$.
\item If $|K:K'| = p^a$, then $Q$ acts faithfully on $K/K'$ and $K/K'$ is irreducible under the action of $Q$.
\end{enumerate}
\end{lem}

\begin{proof}
Let $C = C_K (Z)$.  Since $Z$ acts Frobeniusly on $N$, we have $C \cap N = 1$, and so, $C < K$.  If $C\Phi (K) = K$, then the Frattini subgroup argument would imply that $C = K$, so we have $C \Phi (K) < K$.  By Fitting's theorem, $K/\Phi (K) = C\Phi (K)/\Phi (K) \times [K,Z]/\Phi (K)$.  In particular, when $Q$ does not centralize $Y/X$, we have $\Phi (K) \le X < Y \le [K,Z]$.  We see that $Q$ acts Frobeniusly on $[K,Z]/\Phi (K)$, and so, $Q$ acts faithfully on $Y/X$.  Since $p^a$ is the smallest nontrivial power of $p$ so that $|Q|$ divides $p^a - 1$ It follows that $p^a$ divides $|Y:X|$, and so, $p^a$ divides $|K:K'|$.  Notice that the remaining conclusions follow.
\end{proof}

We will prove several of our results in the slightly more general context of a Gagola group.  We now prove a lemma regarding Gagola groups that applies to those that are not $p$-closed.  In the second author's recent paper, it is proved that if $N$ is the unique minimal normal subgroup that is a $p$-group, then $N < O_p (G)$.  We note that a similar lemma is proved in \cite{nonsolv}.

\begin{lem} \label{cent of N}
Suppose $G$ is a Gagola group that is not a $p$-group with unique minimal normal subgroup $N$ that is a $p$-subgroup.  Let $K = O_p (G)$.  Then the following are true:
\begin{enumerate}
    \item $N \le Z(K)$ and $K = C_G (N)$.
    \item If $K$ is not abelian, then $N \le K'$.
    \item If $G$ is $p$-closed, then $N = Z(K)$.
\end{enumerate}  
\end{lem}

\begin{proof}
We know that $N > 1$ is a normal subgroup of the $p$-group $K$, so $N \cap Z(K) > 1$.  Since $N$ is a minimal normal subgroup of $G$, this implies that $N = N \cap Z(K)$, and we have $N \le Z (K)$.  Thus, $K \le C_G (N)$.  Since $N$ is normal in $G$, we see that $C_G (N)$ is normal in $G$.  If $C_G (N) > K$, then there is a prime $q \ne p$ so that $q$ divides $|C_G (N)|$.   However, if $Q$ is a Sylow $q$-subgroup of $C_G (N)$, then $NQ$ is a Frobenius group which is a contradiction.  Thus, $C_G (N) = K$.  If $K$ is not abelian, then $K' > 1$.  Since $K'$ is characteristic in $G$, we see that $K'$ is normal in $G$.  Because $N$ is the unique minimal normal subgroup of $G$, we conclude that $N \le K'$.  By Lemma 4.2 of \cite{chma}, we know if $G$ is $p$-closed, then $N = Z(K)$.
\end{proof}

Next we are going to show that when $G$ is a Gagola group that is not a $p$-group but whose unique minimal normal subgroup $N$ is a $p$-subgroup, then we have a lower bound on $|O_p (G):O_p (G)'|$ in terms of $|N|$.  Note that this will apply when $|G| = e^4 - e^3$ and in this case, we also obtain an upper bound on the nilpotence class of $O_p (G)$.  

\begin{lem} \label{abelianization}
Let $p$ be a prime and let $a \ge 1$ be an integer.  Suppose $G$ is a Gagola group that is not a $p$-group whose unique minimal normal subgroup has order $p^a$.  Then $|O_p (G):O_p'(G)|$ is divisible by $p^a$.  If $G$ is an Isaacs group of degree $p^a$ and is $p$-closed, then $O_p (G)$ has nilpotence class between $2$ and $a+2$.
\end{lem}

\begin{proof}
Let $K = O_p (G)$.  As we mentioned in Section \ref{background}, Gagola shows in \cite{gagola} that $p^a - 1$ divides $|G:K|$.  
We know that $p^a - 1$ is divisible by a generalized Zsigmondy prime divisor.  Thus, $G$ satisfies the hypotheses of Lemma \ref{prime divisors}, and so, $p^a$ divides $|K:K'|$.  


Now, assume that $|G|= e^4 -e^3$ and $G$ is $p$-closed.  We saw in Section \ref{background} that since $K$ is the Sylow $p$-subgroup of $G$, we have that $(K,N)$ is a Camina pair where $N$ is the unique minimal normal subgroup of $G$, and so $K$ is not abelian.  Also, as mentioned in Section \ref{background} it is shown in \cite{gagola} that $N = Z(K)$, and then using Lemma \ref{cent of N}, we have $N = Z(K) \le K'$.  Once more, we use the result mentioned in Section \ref{background} to see that $|K:Z(K)| = p^{2a}$.  We have shown in Lemma \ref{index ptoa} that $|K':Z(K)| \ge p^a$, so $|K':Z(K)| \le p^a$.  Thus, the upper or lower central series for $P$ can have at most $a$ terms between $P'$ and $Z(P)$, and we obtain the upper bound on the nilpotence class of $a+2$.  Since $P$ is nonabelian, the lower bound on the nilpotence class is $2$.
\end{proof}

For this next lemma, we consider groups that satisfy a slightly weaker hypothesis.  We will see that Gagola groups do satisfy this hypothesis.  Groups that satisfy the hypothesis of this next lemma where every Sylow $q$-subgroup is cyclic or generalized quaternion for $q \ne p$ are called $p'$-semiregular and have been studied in \cite{fleisch}.  We observe that for $p'$-semiregular groups, if a Sylow $p$-subgroup is larger than its index, then the group must have a nontrivial normal $p$-subgroup.  We note that we need some hypotheses for this lemma to be true since $(S_3) {\rm \, \wr \,} Z_3$ has a Sylow $3$-subgroup whose order is larger than its index and the Fitting subgroup of that group is the Sylow $2$-subgroup.  In fact, if $p$ is a Mersenne prime so that $p+1 = 2^n$, then $((Z_2)^n \rtimes Z_p) {\rm \, \wr \,} Z_p$ will have the property that a Sylow $p$-subgroup has order larger than its index and its Fitting subgroup is the Sylow $2$-subgroup.   

We will use $G^\infty$ to denote the {\it solvable residual of $G$}, i.e. the smallest normal subgroup of $G$ whose quotient is solvable, and we write $O (G)$ for $O_{2'} (G)$.

\begin{lem} \label{p prime index}
Let $p$ be a prime.  Assume $G$ is a group so that every Sylow $q$-subgroup for $q \ne p$ is either cyclic or generalized quaternion.  Let $P$ be a Sylow $p$-subgroup of $G$.  If $|P| > |G:P|$, then $O_p (G) > 1$.
\end{lem}

\begin{proof}
We suppose that $O_p (G) = 1$, and prove the contrapositive.  Let $q$ be a prime so that $O_q (G) > 1$.  We know that $q \ne p$ and $O_q (G)$ is cyclic or generalized quaternion.  Suppose first $q$ is odd.  If $q$ is not congruent to $1$ modulo $p$, then $p$ does not divide $|{\rm Aut} (O_q (G))|$.  On the other hand, if $q$ is odd and $q$ is congruent to $1$ modulo $p$, then $|{\rm Aut} (O_q (G))|_p = p^b$ where $p^b$ is the full $p$-power of $q-1$.  If $q = 2$ and $O_2 (G)$ is either cyclic or generalized quaternion but not quaternion, then ${\rm Aut] (O_2 (G))}$ is a $2$-group, and in this case, $p$ will be odd and so $|{\rm Aut} (O_2 (G))|_p = 1$.  The remaining possibility is that $q = 2$ and ${\rm O}_2 (G)$ is the quaternion group of order $8$ and ${\rm Aut} (O_2 (G)) \cong {\rm SL} (2,3)$.  We see that if $p = 3$, then $|{\rm Aut} (O_2 (G)|_3 = 3$ and $|{\rm Aut} (G)|_p = 1$ for all other odd $p$.  In particular, for every prime $q$, we have $|{\rm Aut} (O_q (G))|_p < |O_q (G)|$.  

Let $F(G) = \prod_{i=1}^n O_{q_i} (G)$ for primes $q_1, \dots, q_n$.  Then 
$$G/C_G (F(G)) = G/(\cap_{i=1}^n C_G(O_{q_i} (G))) \le \prod_{1=1}^n G/C_G (O_{q_i} (G))$$ 
and $\prod_{i=1}^n G/C_G(O_{q_i} (G))$ is isomorphic to a subgroup of $\prod_{i=1}^n {\rm Aut} (O_{q_i})$.  Thus, 
$$|G:C_G(F(G))|_p \le prod_{i=1}^n |{\rm Aut} (O_{q_i})|_p < \prod_{i=1}^n |O_{q_i} (G)| = |F(G)|.$$   
If $G$ is solvable, then $C(F(G)) \le F(G)$.  So 
$$|G|_p = |G:F|_p \le |G:C_G(F(G))|_p < |F(G)|,$$ 
and this gives the result.

If $G$ is nonsolvable, then we are going to use Theorem 6.1 of \cite{fleisch}.  Note that Theorem 6.1 of \cite{fleisch} addresses $O^p (G)$.   If $p$ is odd, then Theorem 6.1 (I) of \cite{fleisch} implies that $G^\infty = (O^p (G))^\infty$ is either ${\rm SL} (2,q)$ where $q = p^a$ for some integer $a$ or $p = 3$ or $5$ and $G^\infty$ is ${\rm SL} (2,r)$ for a certain list of primes $r$.  Observe that if $G^\infty \cong {\rm SL} (2,p^a)$, then $p^a < p^a + 1$ and $a_p < p^a-1$.  On the other hand, if $p$ is $3$ or $5$, then $|G^\infty|_p < |G^\infty|_{p'}$.  We see that $|F (G/G^\infty)|_p$ will divide $|{\rm Out} (G^\infty)|_p$ which is $a_p$ when $G^\infty  = {\rm PSL} (2,p^a)$ and is $1$ when $G^\infty = {\rm PSL} (2,r)$.  If $P/G^\infty = O^p (G/G^\infty)$, then we can apply the first paragraph to $G/P$ to see that $|G:P|_p < |G:P|_{p'}$.  Since $|G^\infty| = p^a (p^a+1)(p^a-1)$, we have the result. 

The remaining case is that $p = 2$.  By Theorem 6.1(ii) of \cite{fleisch}, we have that $G^\infty$ is either ${\rm SL} (2,2^n)$, ${\rm Sz} (2^{2m+1})$, or ${\rm SL} (2,2^n) \times {\rm Sz} (2^{2m+1})$ where $n$ and $2m+1$ are relatively prime.  The proof in this case is similar to the proof in the case when $p$ is odd.
\end{proof}

We now use Lemma \ref{p prime index} to show that when $G$ is an Isaacs group of degree $p^a$ that is not $p$-closed, we have $p^{2a} < |O_p (G)| < p^{3a}$. 

\begin{lem} \label{index ptoa}
Let $p$ be a prime and $a \ge 1$ be an integer.  If $G$ is an Isaacs group of degree $p^a$, $N$ is the unique minimal normal subgroup of $G$, and $K = O_p (G)$, then $|K:N| > p^a$.
\end{lem}

\begin{proof}
Suppose that $|K:N| \le p^a$.  Since $|K:N| < p^{2a}$, we see that $G$ is not $p$-closed.  This would imply that $|K| = |K:N||N| \le p^a p^a = p^{2a}$.  On the other hand, we have $|G| = e^4 - e^3 = e^3 (e-1) = (p^a)^3 (p^a -1)$, and so, $p^a (p^a -1)$ divides $|G:K|$.  Take $P$ to be a Sylow $p$-subgroup of $G$.  We have $|P:K| \ge p^a$ is larger than $|G:P| = p^{a-1}$, and by Lemma \ref{p prime index}, we have $O_p (G/K) > 1$ which is a contradiction.
\end{proof}

A corollary of Lemma \ref{index ptoa} is that when $G$ is an Isaacs group of degree $p^a$ that is not $p$-closed, $O_p (G)$ is not abelian.

\begin{cor} \label{Op nonabel}
Let $p$ be a prime and $a \ge 1$ be an integer.  If $G$ is an Isaacs group of degree $p^a$, $N$ is the unique minimal normal subgroup of $G$, and $K = O_p (G)$, then $K$ is nonabelian.    
\end{cor}

\begin{proof}
Suppose $K$ is abelian.  Then $|G:K| \ge d = e^2 - e$, and so, $|K:N| \le e = p^a$.  But this contradicts Lemma \ref{index ptoa}.  Thus, $K$ is not abelian.
\end{proof}

We now use a generalized Zsigmondy prime to obtain  more information about $O_p (G)$.

\begin{lem}\label{Zsigmondy}
Let $p$ be a prime and $a \ge 1$ be an integer.  Suppose that $G$ is an Isaacs group of degree $p^a$ and write $N$ for the unique minimal normal subgroup of $G$.  Suppose $q$ is a generalized Zsigmondy prime divisor of $p^a-1$, let $Q$ be a Sylow $q$-subgroup of $G$, and let $Z$ be the unique subgroup of order $q$ in $Q$.  Write $K = O_p (G)$ and set $C = C_K (Z) N$.
\begin{enumerate}
\item $C$ is normal in $K$ and $C_K (Z) \cap N = 1$.
\item $N \le K'$ and both $K'$ and $C/N$ are abelian.
\item One of the following occurs: 
\begin{enumerate}
    \item $C = N$, $|K| = p^{3a}$, $K$ is a Sylow $p$-subgroup of $G$, $KQ$ is a Frobenius group with Frobenius kernel $K$ and Frobenius complement $Q$, and either $K' = N = Z(K)$ and $|K:N| = p^{2a}$ so $K$ is semi-extraspecial or $N=[K',K] = Z(K)$, $K' = Z_2 (K)$, and $|K:K'| = |K':N| = p^a$.
    \item $K' \le C$, $|K:C| = p^a$, and $C_K (Z) > 1$.
\end{enumerate}
\end{enumerate}
\end{lem}

\begin{proof} 
	We begin by noting in general that $|K| \le p^{3a}$, where quality holds if and only if $G$ is $p$-closed.  It is also the case that $|N| = p^a$ and by Lemma \ref{prime divisors} that $|K:K'|$ is divisible by $p^a$ with $Q$ acting nontrivially on $K/K'$. In the case that $|K| = p^{3a}$, Lemma \ref{cent of N} gives us that $N = Z(K)$.  Otherwise, $|K| < p^{3a}$ and $N \le Z (K)$.  Also note that by Corollary \ref{Op nonabel}, $K$ is known to be nonabelian, and therefore $N \subseteq K'$.

	Next, we show $K/N$ cannot be irreducible.  If $K/N$ is irreducible, then since $N \le K'$ $K$ and $K' < K$, we conclude that $N = K'$. Because $Q$ acts nontrivially on $K/K'$ and $K/N$ is irreducible, it follows by Lemma \ref{prime divisors} that $|K:N| = p^a$. As $|N| = p^a$, we have $|K| = p^{2a}$, which contradicts Lemma \ref{index ptoa}.  Hence, $K/N$ is not irreducible, and in fact, we cannot have $|K:N| \ne p^a$.  
	
	We proceed in cases.  First, assume that $C = N$.  Since $NQ$ is a Frobenius group, we know $C_K(Z) \cap N = 1$, and therefore $C_K(Z) = 1$ in this case.  It follows that $KQ$ is a Frobenius group with Frobenius kernel $K$ and complement $Q$.  Also in this case, observe that $Q$ cannot act nontrivially on any factor of $K$.  If $K/K'$ is irreducible, then $N < K'$ because we have shown $K/N$ cannot be irreducible.  Because $N < K' < K$, $|N| = p^a$, and $Q$ is acting nontrivially on both $K/K'$ and $K'/N$, the orders of both $K/K'$ and $K'/N$ must be divisible by $p^a$ by Lemma \ref{prime divisors}. As $|K| \le p^{3a}$ and we have just shown $|K| \geq p^{3a}$, we must have equality here on all terms.  Thus $|K:K'| = |K':N| = p^a$ and $G$ is $p$-closed.  When $G$ is $p$-closed, we have $N = Z(K)$ and $(K,N)$ is a Camina pair.  Since $N$ is a term in the lower central series and $K'/N$ is irreducible, we obtain $[K',K] = N = Z(K)$, and $K$ has nilpotence class $3$, so $K'$ is abelian.  Since $K/N$ is not abelian and $K/K'$ and $K'/N$ are irreducible, we obtain $K' = Z_2 (K)$.  This proves the theorem in this case.
	
	For the second case, assume that $C > N$.  This is equivalent to $C_K (Q) > 1$.  Let $D/K' = C_{G/K'} (Q)$, and by Fitting's theorem, we have $K/K' = [K,Q]/K' \times D/K'$.   In particular, $K/D$ has no $Q$-invariant factors on which $Q$ acts trivially.  If $K/D$ is not irreducible, then $p^{2a}$ divides $|K:D|$ and $D = N$ which contradicts $C > N$.  Thus, $|K:D| = p^a$ and $|D:K'| \le |D:N| \le p^a$.  We know that $Q$ centralizes some factor of $D/N$.  Since $|D/N| \le p^a$, it follows that $Q$ must centralize every factor of $D/N$.  Since fixed points come from fixed points, we have $D = C_D (Q)N = C_K (Q)N = C$.  Since $K' \le D = C$, we see that $C$ is normal in $K$.  If $K'$ is not abelian, then $K'' > 1$.  Now, $K''$ is characteristic in $K$ which is normal in $G$, so $K''$ is normal in $G$.  Since $N$ is the unique minimal normal subgroup of $G$, we have $N \le K''$.  Thus, $K' = C_{K'} (Q) K''$, and so, $K' = C_{K'} (Q) \Phi (K')$.   By the Frattini subgroup property, this implies $K' = C_{K'} (Q)$ which is a contradiction since $N \le K'$.  Thus, $K'$ is abelian. Let $X/N = \Phi (C/N)$, and observe, that $C = C_K (Q) X$, so $C/N = C_{C/N} (Q) \Phi (C/N)$.  By the Frattini property, $C = C_C (Q)$ which is a contradiction since $N \le C$.  Thus, $C/N$ is abelian.

\end{proof}

We next use the unique subgroup of order $2$ in a Sylow $2$-subgroup when $p$ is an odd prime to obtain structural information.

\begin{lem}\label{2 cent}
Let $p$ an odd prime and let $a \ge 1$ be an integer.  Suppose that $G$ is an Isaacs group of degree $e = p^a$ and write $N$ for the unique minimal normal subgroup of $G$.  Suppose $S$ be a Sylow $2$-subgroup of $G$, take $Z$ be the unique subgroup of order $2$ in $S$, write $K = O_p (G)$, and let $D = C_K (Z)N$.   Then the following are true:
\begin{enumerate}
\item $KZ$ is normal in $G$.
\item $C_K (Z) \cap N = 1$.
\item $G = K N_G (Z)$ and $N_G (Z) \cap K = C_K (Z) > 1$.
\end{enumerate}
\end{lem}

\begin{proof}
We know that $Z$ acts Frobeniusly on $N$, so the nontrivial element in $Z$ inverts every element of $N$ and so, $C_K (Z) \cap N = 1$.  Thus, as an automorphism, $Z$ lies in the center of the automorphism group of $N$.  By Lemma \ref{cent of N}, we have $K = C_G (N)$, and so, $KN/N \le Z(G/N)$.  This implies that $KN$ is normal in $G$.  By the Frattini argument, we have $G = K N_G (Z)$.  Obviously, $C_K (Z) \le N_G (Z) \cap K$.  On the other hand, $[N_G (Z) \cap K,Z] \le Z$ and $[N_G (Z) \cap K,Z] \le [K,Z] \le K$, so $[N_G (Z) \cap K,Z] \le K \cap Z = 1$.  Hence, $N_G (Z) \cap K \le C_K (Z)$, and we have $N_G (Z) \cap K = C_K (Z)$.  In light of Lemma \ref{Op nonabel}, we know that $K$ is nonabelian.  This implies that $KZ$ is not a Frobenius group, and so, $C_K (Z) > 1$.  
\end{proof}




\section{When $e$ is a prime}
 
We now study more specifically the case where $e = p$ is a prime.  We will be able to describe fully what occurs in this case.  Our universal assumption is that $e$ is a prime and $G$ is an Isaacs group so that $|G| = e^4-e^3$.  It is useful also to recall that $d = e^2-e$.  

When $e = 2$, we see that $|G| = 2^4 - 2^3 = 8$.   It is not difficult to see that the two nonabelian groups of order $8$ (the dihedral group of order $8$ and the the quaternion group of order $8$) both satisfy $e= 2$ and $d = 2^2 - 2 = 2$.  Therefore, we may assume that $e \neq 2$ so that $e$ is an odd prime.  

Using the computer software Magma \cite{magma}, we find two groups of order $3^4 - 3^3 = 54$, each having an irreducible character of degree $d = 3^2 - 3 = 6$, and three groups of order $5^4 - 5^3 = 500$ each having an irreducible character of degree $d = 5^2 - 5 = 20$.  For odd primes, we obtain a complete description.  

We begin with strengthened version of Theorem \ref{pclosed} about the general structure of $G$ in the case where $e$ is an odd prime.


\begin{cor}\label{lem:pStructure}
If $G$ is an Isaacs group of degree $e = p$ for an odd prime $p$, then
$G$ can be realized as a semidirect product of an extraspecial $e$-group of order $e^3$ and a cyclic group of order $e-1$, where the cyclic group acts transitively on the nonidentity elements of the center of the $e$-group.  In particular, $G$ must be solvable.
\end{cor}

\begin{proof}
By Theorem 7.2 of \cite{hung2016finite}, we know that since $d,e > 1$ and $|G| = e^4-e^3$ that $G$ has a Gagola character of degree $d$ and a unique minimal normal subgroup $N$ of order $e$.  If $P \in \Syl_e(G)$, Theorem 4.17 of \cite{lewis2018camina} tells us that $Z(P) \subseteq N$, with $Z(P) = N$ if and only if $P \lhd G$.  Since $|G| = e^3(e-1)$, the order of $P$ is $e^3$, and thus we must have $Z(P) = N$ as $Z(P)$ cannot be trivial.  It follows both that $P \lhd G$ and that $P$ is extraspecial.

 
$As$ $G$ must be $p$-closed in this case, Theorem \ref{pclosed} gives us that $(P,Z(P))$ is a Camina pair and that $P$ has a complement $H$ for which $Z(P)H$ is a 2-transitive Frobenius group.  It follows that $N = Z(P)$ has $G = N_G(N)$ and $P = C_G(N)$, therefore $H \cong G/P$ is isomorphic to a subgroup of $\Aut(N)$ by the $N/C$ Theorem.  We conclude that $G/P$ is cyclic and therefore that $G$ is a solvable group as both $P$ and $G/P$ are solvable.  

\end{proof}

We next show that we obtain a unique group in the case when the exponent of the Sylow $e$-subgroup is $e^2$.

\begin{thm}\label{thm:pPrime1}
Let $p$ be an odd prime.  Then there is a unique Isaacs group of degree $p$ whose normal Sylow $p$-subgroup has exponent $p^2$.
\end{thm}

\begin{proof}
Using the structure of $G$ given in Corollary \ref{lem:pStructure}, let $P \in \Syl_e(G)$ and write $G$ as $G = P \rtimes \langle \psi \rangle$, where $\psi \in \Aut(P)$ acts transitively on the nonidentity elements of $Z(P)$.  

We begin by showing that there are maps in $\Aut(P)$ that satisfy the above, so the number of groups $G$ having the character of degree $d$ we desire do occur.  Let $F$ denote the field of order $p$, and suppose $c \in F$ is a generator of the multiplicative group.  Let us view $P$ as being generated by $x,y,$ and $z$, subject to the constraints that $Z(P) = \langle z \rangle$, $[x,y] = z$, $x^p = z$, and $o(y) = o(z) = p$.  The map $\psi$ defined by $\psi(x) = x^c$, $\psi(y) = y$, and $\psi(z) = z^c$ belongs to $\Aut(P)$ and acts as is required by Lemma \ref{lem:pStructure}.  If we were to work in $P \rtimes \langle \psi \rangle$, note that the action of $\psi$ on $\Irr(N)$ is isomorphic to its action on $N$, so there exists a nonprincipal $\lambda \in \Irr(N)$ for which $\lambda^{\psi} = \lambda^c$.  As $N$ is central in $P$ and $\psi$ acts nontrivially on $\lambda$, $P$ is the stabilizer of $\lambda$ in $G$.  If $\theta \in \Irr(P)$ and $\chi = \theta^G$, then $\chi \in \Irr(G)$, and taking degrees, we see $\chi(1) = |G:P|\theta(1) = (e-1)e = d$.  Thus groups having the structure required by Lemma \ref{lem:pStructure} do exist. Because $\psi^k(z) = z^{c^k}$ for all integers $k \geq 1$, different choices of $c$ as an exponent for $z$ belong to $\langle \psi \rangle$.  Thus different choices for the exponent on $x$ and $z$ do not lead to non-isomorphic groups $G$. 


Our goal now is to show that any other $\psi \in \Aut(G)$ is ultimately similar to one described in the previous paragraph, proving the uniqueness.  That is, given $\psi \in \Aut(G)$ satisfying the requirements of Lemma \ref{lem:pStructure}, we will show that there exist group elements $x,y,$ and $z$ for which  $\psi(x) = x^c$, $\psi(y) = y$, $\psi(z)  = z^c$, $[x,y] = z$, and $x^p = z$, and $o(z) = o(y) = p$.  

To begin, let $c \in F^{\times}$ be such that $\psi(z) = z^c$ for some $z \in Z(P)$, noting that $\psi(z^m) = \psi(z)^m = (z^{mc})  = (z^m)^c$.  That is, if $\psi$ maps $z$ to $z^c$ for some $z$, then $\psi$ maps $z$  to $z^c$ for all $z \in Z(P)$. 

Next, we show that $C_P(\psi)$ is a subgroup of $P$ of order $p$ fixed by $\psi$.  We use the bar convention to denote modding by $Z(P)$.  By Corollary 1 of \cite{winter}, there exists a nontrivial $\bars{y} \in \bars{P}$ left fixed by every element of $\Aut(P)$.  In particular, $C_{\bars{P}}(\psi)$ is non-trivial.  Since $\psi$ acts nontrivially on $Z(P)$, we claim that $C_{\bars{P}}(\psi)$ is also proper.  To see this, suppose $g \in P$ is such that $\psi(\bars{g}) = \bars{g}$, so that $\psi(g) = gv$ for some $v \in Z(P)$.  Then $\psi(g^p) = \psi(g)^p = (gv)^p = g^pv^p = g^p \, \, ,$
where the penultimate equality holds because $v \in Z(P)$.  Hence $g^p$ is fixed by $\psi$, and as $g^p$ is central, we must have $g^p = 1$ in such circumstances.  We conclude from this
 that $C_{\bars{P}}(\psi)$ is proper in the case that $P$ has exponent $p^2$.  We further conclude that the preimage of $C_{\bars{P}}(\psi)$, henceforth $A$, is a normal elementary abelian subgroup of $G$ having order $p^2$.  By Corollary 3.28 of \cite{isaacs2008finite}, $C_{\bars{P}}(\psi) = \bars{C_P(\psi)}$.  Because $A$ is elementary abelian, contains $Z(P)$, and $\psi$ acts nontrivially on $Z(P)$, we conclude that $C_P(\psi)$ is a subgroup of $A$ disjoint from $Z(P)$ and complementing it in $A$.  Note that this is situation (c) of Lemma \ref{Zsigmondy}.  

 Because $\bars{P}$ is abelian, we can apply Fitting's Theorem to conclude $\bars{P} = [\bars{P},\psi] \times \bars{A}$.  Letting $B$ be the preimage of $[\bars{P},\psi]$, we see $B$ is a subgroup of $P$ of order $p^2$ normalized by $\psi$ with $B \cap A = Z(P)$.  That $B$ must be cyclic of order $p^2$ comes from the fact that $A$ is the kernel of the $p$-power map, a homomorphism for an extra-special $p$-group, and therefore the $p^2(p-1)$ elements of $G$ not in $A$ must all have order $p^2$.

 To finish our proof, let $x$ be a generator of $B$.  Since $B$ is normalized by $\psi$, we must have $\psi(x) = x^m$ for some $m$, and therefore $\psi$ is the $m$-th power map on all of $B$.  Because $x^p$ is central and contained in $B$, and $\psi$ is already known to be the $c$-th power map on the center of $P$, we conclude that $\psi(x) = x^c$.  Letting $y \in A$ be non-trivial and taking $z = [x,y]$ yields the representation we desire.



\end{proof}

Corollary \ref{lem:pStructure} gives us that $G$ has a normal extraspecial Sylow $e$-subgroup of order $e^3$.  Up to isomorphism there are exactly two extraspecial $e$-groups order $e^3$, and Theorem \ref{thm:pPrime1} part (b) handles the case when the extraspecial $e$-group has exponent $e^2$.  Our second theorem in this section handles when the extraspecial $e$-group has exponent $e$.

\begin{thm}\label{thm:pPrime2}
Let $p$ be an odd prime.  Then there exist $\dfrac{p-1}{2}$ non-isomorphic Isaacs groups of degree $p$ whose Sylow $p$-subgroup has exponent $p$.
\end{thm}

\begin{proof}
By Corollary \ref{lem:pStructure}, $G$ is a solvable group with $P \in \Syl_e(G)$ being both normal in $G$ and an extraspecial $e$-group of order $e^3$. This theorem also gives us that $G$ has a cyclic $e$-complement, so it is possible to view $G$ as a semidirect product of $P$ and an automorphism of $P$.  In order for $G$ to have an irreducible character of degree $d$, all $p-1$ non-identity elements of $\Irr(Z(P))$ lying under an irreducible character of $G$ of degree $d$ must lie in a single orbit.  It follows that there must exist $\psi \in \Aut(P)$ placing all $p-1$ non-identity elements of $Z(P)$ in a single orbit.  We count the number of total groups $G$ having the properties described in this theorem by counting the number of conjugacy classes of $\langle \psi \rangle$ in $\Aut(P)$ where $\psi$ has order $p-1$ and acts on $Z(P)$ in the necessary way. 

We claim that whenever $\psi \in \Aut(P)$, there is a corresponding element of $\Aut(P/Z(P)) \cong \textnormal{GL}(2,p)$ coming from the action of $\psi$ on $x$ and $y$.  Let $g \in P$ and use the bar convention to denote the image of $g$ modulo $Z(P)$.  For $\psi \in \Aut(P)$, define $\bars{\psi} \in \Aut(\bars{P})$ as $\bars{\psi}(\bars{g}) = \bars{\psi(g)}$.  Note that $\bars{\psi}$ is surjective because $\psi$ is surjective, and $\bars{\psi}$ is a homomorphism because 
$$\bars{\psi}(\bars{g} \cdot \bars{h}) = \bars{\psi}(\bars{g h}) = \bars{\psi(g h)} = \bars{\psi(g)} \cdot \bars{\psi(h)} = \bars{\psi}(\bars{g}) \cdot \bars{\psi}(\bars{h}) \, \, .$$

To see that $\bars{\psi}$ is injective, assume that $\bars{\psi}(\bars{g}) = \bars{\psi}(\bars{h})$.  Then $\bars{\psi{(g)}} = \bars{\psi(h)}$, and this combined with the fact that $\psi$ is surjective implies that there exists $v \in Z(P)$ for which $\psi(g) = \psi(h)\psi(v)$.  Because $\psi$ is a homomorphism $\psi(g) = \psi(hv)$, and as $\psi$ is injective, $g = hv$.  It follows that $\bars{g} = \bars{h}$, and injectivity of $\bars{\psi}$ holds.  Thus $\bars{\psi} \in \Aut(\bars{P})$, which is isomorphic to GL$(2,p)$.  In what follows, we use this correspondence and our understanding of elements of GL$(2,p)$ to classify elements of $\Aut(P)$.

Write $P = \langle x, y, z \rangle$, where $z$ is a generator of $Z(P)$ and $x$ and $y$ are such that $[x,y] = z$.  We assume that the exponent of $P$ is $e$, so that $x^p = y^p = 1$.  

Let $F$ denote the field having $p$ elements.  Because of the way $\psi$ must act on $Z(P)$, there must exist $c \in F^{\times}$ a generator of $F^{\times}$ for which $\psi(z) = z^c$.  For $a \in F^{\times}$, there is a unique $b \in F^{\times}$ for which $ab = c$ in $F^{\times}$, and we claim that the map $\psi$ defined by $\psi(x) = x^a$ and $\psi(y) = y^b$ will satisfy $\psi(z) = z^c$. To see this, recall that in any group we have the commutator identities $[g,hl] = [g,h][g,l]^h$ and $[gh,l] = [g,l]^h[h,l]$.  In a group where the commutator is central, the conjugation action will be trivial.  Thus in an extraspecial $p$-group, we will have $[g,hl] = [g,h][g,l]$ and $[gh,l] = [g,l][h,l]$.  Applying these general identities with $[g,y^b]$ and $[x^a,g]$, we see that $[x^a,y^b] = [x,y]^{ab}= z^c$, as claimed.  From this identity, it follows that if $\psi$ sends $x$ to $x^a$ and $y$ to $y^b$, $\psi$ will also send $z$ to $z^c$; in fact, any $\psi$ that acts in this way will have the properties we desire.  

For $\psi$ chosen in this way, the corresponding element of GL$(2,p)$ is the matrix $\left( \begin{array}{lr} a & 0 \\ 0 & b \end{array}\right)$.  Note that this matrix is conjugate to $\left( \begin{array}{lr} b & 0 \\ 0 & a \end{array}\right)$  via the matrix $\left( \begin{array}{lr} 0 & 1 \\ 1 & 0 \end{array}\right)$, so the map sending $x$ to $x^b$ and $y$ to $y^a$ is ultimately conjugate to $\psi$ in $\Aut(P)$.  Since similar matrices have the same eigenvalues, different choices of $a$, other than swapping the roles of $a$ and $b$, will not result in conjugate automorphisms. 

In summary, if we choose $c \in F^{\times}$ with $c$ a generator of the multiplicative group and choose an arbitrary $a \in F^{\times}$, we create an element of $\Aut(P)$ so that when we consider $P \rtimes \langle \psi \rangle$ we have a group $G$ with an irreducible character of degree $d$.  Let us count such groups by considering how many conjugates $\langle \psi \rangle$ has in $\Aut(P)$.

Because $\psi^k(z) = z^{c^k}$ for all integers $k \geq 1$, different choices of $c$ as an exponent for $c$ belong to $\langle \psi \rangle$.  Hence the number of choices for $c$ is not important for the total number of non-isomorphic groups $G$.  There are $p-1$ total choices for the member of $F^{\times}$ that serves as an exponent for $x$, and we have established that the only other choice for $a$ that leads to a conjugate member of $\Aut(P)$ is the corresponding $b$. Consequently, there are $\dfrac{p-1}{2}$ total conjugacy classes of subgroups of $\Aut(P)$ that can be built in this way.

To complete the proof, we must now show that any $\psi \in \Aut(P)$ that acts transitively on the nontrivial elements of $Z(P)$ is similar to one we have just described.  To do this, let us assume $\psi \in \Aut(P)$ acts transitively on the nontrivial elements of $Z(P)$, and let us consider the matrix $\bars{\psi}$ of GL$(2,p)$ that corresponds to $\psi$.  By \cite{piatetski1983complex}, there are four main types of conjugacy class of matrix to consider, where these four types are generally classified by their eignvalues:

\begin{enumerate}
\item $\left( \begin{array}{lr} a & 0 \\ 0 & a \end{array}\right)$, where $a \in F^{\times}$ 
\item $\left( \begin{array}{lr} a & 0 \\ 0 & b \end{array}\right)$, where $a \neq b,$ $a,b \in F^{\times}$
\item $\left( \begin{array}{lr} a & 1 \\ 0 & a \end{array}\right)$, where $a \in F^{\times}$
\item $\left( \begin{array}{lr} a & b\alpha \\ b & a \end{array}\right)$, where $a$ and $b$ are not both zero, $\alpha \in F^{\times}$ and $\alpha$ is a non-square. 
\end{enumerate}

If $\bars{\psi}$ belongs to a conjugacy class of the first form, then $\psi$ is similar to the automorphism of $P$ that sends $x$ to $x^a$ and $y$ to $y^a$, thus sending $z$ to $z^{2a}$.  Such an automorphism cannot act transitively on the nontrivial elements of $Z(P)$, and therefore cannot produce an automorphism capable of producing a group with an irreducible character of degree $d$.

If $\bars{\psi}$ is of the form $\left( \begin{array}{lr} a & 1 \\ 0 & a \end{array}\right)$, where $a \in F^{\times}$, note that for integers $k > 1$ we have 
\begin{equation}\label{eqn:matrixPower}
\left(  \begin{array}{lr} a & 1 \\ 0 & a \end{array}\right)^k = \left( \begin{array}{lr} a^k &  ka^{k-1}\\ 0 & a^k \end{array}\right) \, \, .
\end{equation}

Formula (\ref{eqn:matrixPower}) shows us that the order of such a matrix must necessarily be divisible by $p$.  This implies that the order of any corresponding automorphism $\psi$ must be divisible by $p$, which will not produce an automorphism of $P$ that will result in a group $G$ having an irreducible character of degree $d$.

Next, suppose $\bars{\psi}$ is of the form $\left( \begin{array}{lr} a & b\alpha \\ b & a \end{array}\right)$, where $a$ and $b$ are not both zero and $\alpha \in F^{\times}$ is a non-square.  Note that there is an isomorphism between matrices of this form and the multiplicative group of $F[\sqrt{\alpha}]$ via $\left( \begin{array}{lr} a & b\alpha \\ b & a \end{array}\right) \mapsto a + b \sqrt{\alpha}$.  As the multiplicative group of $F[\sqrt{\alpha}]$ is a cyclic group, it has a unique subgroup of order $p-1$, and this unique subgroup corresponds to the multiplicative group of the ground field $F$.  It follows that any matrix with order $p-1$ necessarily has $b = 0$.  However, this results in matrices of the first form above, and we have already shown that such matrices do not correspond to automorphisms of $P$ that act on $Z(P)$ in a way that will ultimately produce an irreducible character of $G$ having order $d$.

By process of elimination, this shows any automorphism of $P$ that ultimately leads to an irreducible character of $G$ of degree $d$ has a corresponding matrix of the second form; that is, any such automorphism will be conjugate to one sending $x$ to $x^a$, $y$ to $y^b$ and $z$ to $z^c$, where $a,b \in F^{\times}$ and $ab = c$, where $c$ is a generator of the multiplicative group $F^{\times}$.
\end{proof}

\section{Results for $e = p^2$} \label{p squared}

In this section, we now specialize to the case where $e = p^2$ for a prime $p$.  We first determine the possible structures for $G$ when $G$ is $p$-closed.

\begin{thm}\label{p^2 pclosed}
Let $p$ be a prime.  If $G$ is an Isaacs group of degree $p^2$ with $G$ $p$-closed, then the normal Sylow $p$-subgroup $K$ of $G$ is a $p$-group of order $p^6$ with $Z(K)$ elementary abelian of order $p^2$ and either:
\begin{enumerate}
\item $K' = Z(K)$ so that $K$ is semi-extraspecial or
\item $Z(K) = [K',K] < Z_2 (K) < K'$ so that $Z_2 (K)$ has order $p^4$ and $K'$ has order $p^3$.    
\end{enumerate} 
\end{thm}

\begin{proof}
By Theorem \ref{pclosed}, we know that $K$ is a normal subgroup of order $p^6$ and $(K,N)$ is a Camina pair with $N = Z(K)$.  This implies that $N$ is an elementary abelian $p$-group of order $p^2$ and $\cd {K \mid N} = \{ p^2 \}$.  If $K' = Z(K)$, then $K$ is semi-extra-special and we are done.  Otherwise, we must have $Z(K) < K'$.  We know that $p^3 \le |K:Z_2(K)| \ge p^2$.  If $|K:Z_2 (K)| = p^2$, then $|(K/Z(K))'| = p$, and so, $|K':Z(K)| = p$.  This implies that $Z(K) = [K',K] < Z_2 (K) < K'$ so that $Z_2 (K)$ has order $p^4$ and $K'$ has order $p^3$.  

Now, suppose $|K:Z_2 (K)| = p^3$.  Note that if $K' = Z_2 (K)$, then $K/Z(K)$ is extra-special since $|Z(K/Z(K))| = |(K/Z(K))'| = p$, and this is a contradiction since $|K:K'| = p^3$ is not a square.  Thus, we must have $|K:K'| = p^2$.  This implies $K' = Z_3 (K) > K_3 = [K',K] = Z_2 (K) > K_4 = [K_3,K] = Z(K)$.  Observe that $|K':K_3| = |K_3:K_4| = p$. Take $q$ to be an odd prime divisor of $p+1$ if such a prime exists.  Otherwise, $p+1$ is a power of $2$, and take $q = 2$.  Write $Q$ for a Sylow $q$-subgroup of $G$.   When $q = 2$, take $Z$ to be the unique subgroup of order $2$ in $Q$.  Note that we may apply Lemma \ref{prime divisors} with $q$.   Notice that $K_2$, $K_3$, and $K_4$ are characteristic in $K$ and so are fixed by $Q$.  Since $|K_2/K_3|$ and $|K_3/K_4|$ are less than $p^2$, we see that $Q$ must centralize $K_2/K_3$ and $K_3/K_4$.  Because fixed points come from fixed points we can find $b \in C_{K_2} (Q)$ so that $K_2 = \langle b, K_3 \rangle$.  Consider the map $K/K' \rightarrow K_3/K_4$ by $aK' \mapsto [a,b]K_4$ for $a \in K$.  Since $[K_2,b] \subseteq K_4$, it is not difficult to see that this a well-defined surjective homomorphism.  However, $Q$ will act transitively on $K/K'$ and cannot act transitively on $K_3/K_4$, so we have a contradiction.  Thus, this case cannot occur.
\end{proof}

Verardi has proved in \cite{ver} that there is a unique isoclinism class of ultraspecial groups of order $p^6$ for every prime $p$.  Hence, for each odd prime $p$ there is a unique ultraspecial group of order $p^6$, the group that is isomorphic to a Sylow $p$-subgroup of ${\rm GL}_3 (p^2)$.  Using the description of the automorphism group of this group in \cite{lewiswilson}, we will obtain a number of extensions that yield $e = p^2$.  From the examples, we have constructed for $p = 3, 5, 7,$ and $11$, it also appears that there are two isoclinic groups of exponent $p^2$ whose automorphism group has order divisible by $p^2 - 1$.  For $p = 3$ and $p = 5$, both of these groups also yielded extensions with $e = p^2$.  

Finally, we see that there is also a possible Camina pair that may have extensions with $e = p^2$.  For $p = 3$ and $p = 5$, we have been able to construct extensions.  When $p \ge 5$, in order for the group $G$ to have an automorphism of order $p^2-1$, it appears to have to have exponent $p$ and it also appears that there is up to isomorphism, a unique such group of exponent $p$ for each prime $p \ge 5$.  

We now turn to the non $p$-closed examples, and  we now describe how to find examples for $p^2$.

\begin{thm} \label{p^2 not pclosed}
Let $p$ be a prime.  Suppose $G$ is an Isaacs group of degree $p^2$ with unique minimal normal subgroup $N$ where $G$ not $p$-closed, then the following hold:
\begin{enumerate}
\item $G$ has a normal $p$-subgroup $K$ of order $p^5$ so that $Z(K) = N$ and either:
\begin{enumerate}
    \item $K' = Z (K) = N$ (but not semi-extraspecial) or 
    \item $K' = Z_2 (K)$ has order $p^3$ and $Z(K) = [K',K] = N$.
\end{enumerate}
\item $G/K \cong {\rm SL} (2,p)$ and $G/K$ acts on $N$ as its natural module.
\item If $p$ is odd, take $S$ to be a Sylow $2$-subgroup of $G$ and $Z$ the unique subgroup of order $2$ in $S$, then $G = K N_G (Z)$ where $G/K \cong {\rm SL} (2,p)$ and $N_G (Z) \cap K = C_K (Z)$ is cyclic of order $p$.  Taking $C = C_K (Z)N$, we have $K' \le C$, $C$ is normal in $G$, and $C$ is abelian.
\item If $\lambda \in \Irr (N)$, then $G$ acts transitively on $\Irr (K \mid \lambda)$.
\end{enumerate}
\end{thm}

\begin{proof}
Since $G$ is not $p$-closed, we see that $|K| < (p^2)^3 = p^6$.  By Lemma \ref{index ptoa}, we have $|K:N| > p^2$, so $|K| > p^4$.  This implies that $|K| = p^5$.  Using Corollary, \ref{Op nonabel}, we know that $K$ is nonabelian, so $K' \ge N$.  Also, in light of Lemma \ref{prime divisors}, we have $p^2$ divides $|K:K'|$.  Hence, either $|K:K'| = p^2$ or $K' = N$.  By Lemma \ref{cent of N}, we have $N \le Z(K)$.  We know $|K:Z (K)| \ge p^2$.  If $|K:Z(K)| = p^2$, then $|K'| = p$, and this contradicts $N \le K'$.  Thus, $Z(K) = N$.   

If $|K:K'| = p^2$, then since $K' \not\le Z(K) = N$, we have $[K',K] > 1$.  Since $[K',K]$ is characteristic in $K$ which is normal in $G$, we have $[K',K]$ is normal in $G$, and since $N$ is the unique minimal normal subgroup of $G$, we have $N \le [K',K] < K'$.  Since $|K':N| = p$, we deduce that $|[K',K] = N = Z(K)$.  Since $K/N$ is a nonabelian group of order $p^3$, we obtain $K' = Z_2 (K)$.  The other possibility is that $K' = N = Z(K)$.  Since $|K:K'| = |K:N| = p^3$ is not a square, we know $K$ is not semi-exxtraspecial.

We know that $|N| = p^2$ and $N$ is elementary abelian of order $p^2$.  Thus, ${\rm Aut} (N) = {\rm GL} (2,p)$.  By Lemma \ref{cent of N}, we have that $K = C_G (N)$.  We know that $G/K = G/C_G (N)$ is isomorphic to a subgroup of ${\rm Aut} (N) = {\rm GL} (2,p)$.  Let $X$ be the image of $G/K$ in ${\rm GL} (2,p)$ We know that $|G:K| = p (p^2-1)$.  We also know that $|{\rm GL} (2,p): {\rm SL} (2,p)| = p-1$.  Thus, a Sylow $p$-subgroup of $X$ is contained in ${\rm SL} (2,p)$.  If $p+1$ a has an odd prime divisor $q$, then ${\rm SL} (2,p)$ will be generated by a Sylow $q$-subgroup of $X$.    If $p+1$ is a power of $2$, then take $q = 2$ and $X$ will have an element of order $p+1$ and this element is conjugate to an element of ${\rm SL} (2,p)$.  We see that the image of $G/K$ will contain the center of $G/K$ and so, we can use Dickson's classification of subgroups of ${\rm PSL} (2,p)$ to see that ${\rm PSL} (2,p)$ is the only subgroup that contains either both a Sylow $p$-subgroup and a Sylow $q$-subgroup for a Zsigmondy prime $q$ or both a Sylow $p$-subgroup and an element of order $(p+1)/2$ when $p+1$ is a power of $2$.  It follows that $G/K$ is isomorphic to ${\rm SL} (2,p)$.  Since $G/K$ acts transitively on $N$, it follows $N$ is the natural module for $G/K$.

Suppose $p$ is odd.  We may apply Lemma \ref{2 cent} to see that $G = K N_G (Z)$, $N_G(Z) \cap K = C_K (Z)$, and $C_K (Z) > 1$.  Let $q$ be an prime divisor of $p+1$ if one exists and $q=2$ if $p+1$ is a power of $2$, and let $Q$ be a Sylow $q$-subgroup of $G$ contained in $N_G (Z)$.  Observe that $Q$ normalizes $C = C_G (Z)N$.  If $K' = N$, then $K' \le C$.  If $|K:K'| = p^2$, then $Q$ acts irreducibly on $K/K'$.  Hence, we have $K = CK'$ or $C = K'$.  Suppose $K = CK'$, and note that this implies that $K = C \Phi (K) = C_K(Z) N\Phi (K) = C_K (Z) \Phi (K)$, and using the Frattini property, we have $K = C_K (Z)$ which is a contradiction.  Thus, $C = K'$.  Thus, in both cases, we have $K$ normalizes $C$ and so, $G = K N_G (Z)$ normalizes $C$.  Since $Q$ acts irreducibly on $K/C$, we see that $|K:C| = p^2$ and $|C:N| = |C_K (Z)| = p$.

We know that that $\lambda^G$ has a unique irreducible constituent, so $G$ acts transitively on $\Irr (K \mid \lambda)$.
\end{proof}

Continuing in the situation of Theorem \ref{p^2 not pclosed}, when $p$ is odd, we have that $G/C$ is isomorphic to ${\rm SL}_2 (p)$.  Also, for $p \ge 5$, we know that ${\rm SL}_2 (p)$ is not solvable, so as we noted in the Introduction, if an example exists, then it will not be solvable.

\section{Examples}

We now include examples that show that when we drop the assumption that $e = p$ that we can find examples where $P \in \Syl_p(G)$ is not normal and examples where $P \in \Syl_p(G)$ is not semi-extraspecial.  

\subsection{Isaacs groups of degree $4$}

When $e = 4$, we able obtain a complete description of the Isaacs groups of degree $4$ using the computer algebra system Magma (or GAP) (see \cite{magma} or \cite{GAP}).  Note that when $e = 4$, we have that $|G| = 4^4 - 4^3 = 192$ and $d = 4^2 - 4 = 12$.  Up to isomorphism, we have six different groups.  Four of these groups are $2$-closed and two are not.  We obtain three different multisets of character degrees each of which occurs for two groups.  

For the groups that are $2$-closed, two of the groups have character degree multiset $[<1,12>,<3,4>,<12,1>]$ and two of the groups have character degree multiset $[<1,3>,<3,5>,<12,1>]$.  Notice that the second pair of groups will be Frobenius groups.   The Frobenius kernel is $P$ and a Frobenius complement is $Z_3$.  (This cannot occur when $e$ is odd, since $2$ divides $e-1$ and we know that if $2$ divides the order of the Frobenius complement, then the Frobenius kernel must be abelian; and this would contradict the fact that the normal Sylow subgroup must satisfy that $(P,Z(P))$ is a Camina pair.)  In fact, the normal Sylow $2$-subgroup in this case has character degree multiset $[<1,16>,<4,3>]$ which implies that $P$ is semi-extraspecial.  We note that there are up to isomorphism three different Sylow $2$-subgroups in these two cases and one of the $2$-groups occurs as Sylow $2$-subgroups for both pairs.  All three of the Sylow $2$-subgroups have exponent $4$.

We also have two groups that are not $2$-closed.  Both of these groups have character degree multiset $[<1,4>,<2,2>,<3,4>,<12,1>]$.  The Sylow $2$-subgroups for these groups both have character degree multiset $[<1,8>,<2,2>,<4,3>]$.  One of the Sylow $2$-subgroups has exponent $4$ and the other one has exponent $8$, so they are definitely not isomorphic.  The one with the Sylow $2$-subgroup having exponent $4$ is the example of order $64*3 = 192$ that is described in Section 6 of \cite{gagola}.

\subsection {Isaacs groups of degree $9$}

Note that when $G$ is an Isaacs group of degree $9$, we have $|G| = 9^4 - 9^3 = 5832$ which is too large to be in the small groups library.  Also, $d = 9^2 - 9 = 72$.  On the other hand, a Sylow $3$-subgroup of $G$ has order $9^3 = 3^6 = 729$, and groups of this order are in the Small Group library.  Thus, we can find the $3$-closed groups by searching through the groups of order $729$ whose center has order $9$, whose set of irreducible characters contain $8$ characters of degree $9$, and whose automorphism group contains either an element of order $8$ or a subgroup of the automorphism group that is isomorphic to the quaternions of order $8$ that acts Frobeniusly on the center.  If $P$ is such a $3$-group and $Q$ is the stated subgroup of order $8$, then taking $G$ to be the semi-direct product of $Q$ acting on $P$ will give us our group.

Using Magma \cite{magma}, we find the groups of order $729$ whose center has order $9$, that have an irreducible character of degree $9$, and whose automorphism group has order divisible by $8$.   We find six groups that satisfy these conditions.  In the Small Groups library in Magma (or GAP) they are: $225, 230, 235, 469, 472, 473$.  The first three ($225, 230, 235$) have nilpotence class $3$ and exponent $9$.  We note that the Automorphism group of ${\rm SmallGroup} (729,225)$ is the dihedral group of order $8$, so it does not produce an example.  The ${\rm SmallGroup} (729,469)$ is the Heisenberg group of order $729$ and the remaining two groups, which are ${\rm SmallGroup} (729,472)$ and ${\rm SmallGroup} (729,473)$, are isoclinic to the Heisenberg group having exponent $9$.  

Now, the automorphism group of $P = {\rm SmallGroup} (729,469)$ has four nonconjugate cyclic subgroups of order $8$, and so, we obtain four nonisomorphic Isaacs groups $G$ that are semi-direct products of $P$ acted on by a cyclic group of order $8$ that act Frobeniusly on the center.  The automorphism groups for the groups, 
$${\rm SmallGroup} (729,235), {\rm SmallGroup} (729,472), {~\rm and ~} {\rm SmallGroup} (729,473),$$  each only have one conjugacy class of cyclic subgroups of order $8$ that act Frobeniusly on the center, and so up to isomorphism, we only get one group $G$ from each of them.  Hence, we obtain up to isomorphism seven groups of order $5832$ that are $3$-closed where the Sylow $2$-subgroup is cyclic.
 
The Isaacs group $G$ that has as a normal Sylow $3$-subgroup the one that is not a semi-extraspecial group has character degree multiset of 
$$[ <1, 8>, <2, 4>, <6, 12>, <8, 3>, <72, 1> ].$$  
For the Isaacs groups that have a normal Sylow $3$-subgroup that is semi-extraspecial, we see that the  character degree multiset 
$$[ <1, 8>, <4, 4>, <8, 9>, <72, 1> ]$$ occurs for three groups (two of these have the Heisenberg group as their Sylow $3$-subgroup), the character degree multiset 
$$[ <1, 72>, <8, 9>, <72, 1> ]$$ occurs for two groups with different Sylow $3$-subgroups, and finally, the character degree multiset 
$$[ <1, 8>, <2, 16>, <8, 9>, <72, 1> ]$$ 
occurs for one group.  

We now turn to the examples with the quaternion group of order $8$.  The automorphism group of the Heisenberg group, ${\rm SmallGroup} (729,469)$ has two conjugacy classes that are isomorphic to the quaternion group that act Frobeniusly on the center.  The automorphism groups for 
$${\rm SmallGroup} (729,230), {\rm SmallGroup} (729,235), {\rm SmallGroup} (729,472), $$
$${~\rm and ~} {\rm SmallGroup} (729,473)$$ 
each contain one conjugacy class of subgroups isomorphic to the quaternions that act Frobeniusly on the center.  Therefore, we obtain six Isaacs groups $G$ in this way.

The Isaacs groups $G$ coming from both non semi-extraspecial groups with the quaternions have character degree multiset that equal 
$$[ <1, 4>, <2, 5>, <3, 8>, <6, 10>, <8, 3>, <72, 1> ].$$  For the semi-extraspecial groups with the quaternions, we get two groups having the character degree multiset 
$$[ <1, 12>, <2, 15>, <8, 9>, <72, 1> ]$$ and two groups having the character degree multiset 
$$[ <1, 4>, <2, 9>, <4, 2>, <8, 9>, <72, 1> ].$$  

To determine the non $3$-closed Isaacs groups of degree $9$, we have Magma find the groups of order $243$ whose centers have order $9$, whose derived subgroups contain the center and have order either $9$ or $27$, and whose automorphism group have order divisible by $8$.  We find five groups where the center has order $9$ and two groups where the center has order $3$.  From these, we find that there is only one non $3$-closed  Isaacs group of degree $9$.  This group has $K = O_3 (G)$ with $K' = Z(K)$ of order $9$.  Note that this will be the example of order $3^6*2^3 = 5832$ in Section 6 of \cite{gagola}.

\subsection {Isaacs groups of degree $25$}

The first part of the procedure for finding the Isaacs groups of degree $25$ is similar to that for finding those of degree $9$.  We have $|G| = 25^4 - 25^3 = 375,000$ and $d = 25^2 - 25 = 600$.   Again, the groups $G$ are too big to be in the Small Groups library, but a Sylow $5$-subgroup of $G$ has order $25^3 = 5^6 = 15,625$, and groups of this order are in the Small Group library.  Thus, we can find the $5$-closed groups by searching through the groups of order $5^6$ whose center has order $25$, whose set of irreducible character degrees contain $24$ characters of degree $25$, and whose automorphism group contains either an element of order $24$ or a subgroup of the automorphism group that is isomorphic to the cyclic group of order $8$ acting nontrivially on a group that acts Frobeniusly on the center.  None of these groups have the an automorphism group with the quaternions as a subgroup, so we do not obtain any complements containing the quaternions.

As in the case when $e = 9$, we obtain the Heisenberg group of order $5^6$, ${\rm Small Group} (5^6,276)$, and two other semi-extraspecial groups of exponent $25$ that are isoclinic to the Heisenberg group: 
$${\rm Small Group} (5^6,278) {~ \rm and ~}{ \rm Small Group} (5^6,279).$$  
We also obtain one group that has nilpotence class $3$ and exponent $5$: 
$${\rm Small Group} (5^6,498).$$

In the automorphism group of the Heisenberg group, there are twelve conjugacy classes of cyclic subgroups of order $24$ that act Frobeniusly on the center and six conjugacy classes of subgroups that are isomorphic to a cyclic group of order $8$ acting nontrivially on a group of order $3$ that act Frobeniusly on the center.  For each of the other three groups, their automorphism groups each contain only one conjucgacy class of each type of subgroup.  Thus, we obtain $15$ Isaacs groups that are $5$-closed where a Hall $5$-complement is cyclic, and $9$ Isaacs groups that are $5$-closed where a Hall $5$ complement is cyclic group of order $8$ acting nontrivially on a cyclic group of order $3$.

Since we have $10$ different multisets of character degrees that occur for the groups with a cyclic Hall $5$-complement and $6$ different multisets of character degrees that occur for the groups with a noncyclic Hall $5$-complement, we are not going to list them all here.

To find the non $5$-closed Isaacs groups of degree $25$, we use Magma to search the groups of order $3,125$ that a center of order $25$ and either have the derived subgroup equal to the center or have the derived subgroup of order $125$ and contain the center.   In both cases, we then also look for groups whose automorphism group has order divisible by $24$.  We find that there are five group with nilpotence class $2$ and two groups of nilpotence class $3$.  Constructing the extensions, we find one Isaacs group $G$ of degree $25$ that is not $5$-closed which has $O_5 (G)$ has nilpotence class $3$, and this group has appeared in \cite{nonsolv}, and we find one Isaacs group $G$ of degree $25$ that is not $5$-closed which has $O_5 (G)$ has nilpotence class $2$, and we will give generators and relations below for this group.

\begin{verbatim}
G = < a,b,c,d,e,f,g,h | a^5, b^5, c^5, d^5, e^5, f^4, g^3, h^2, 
[b, c], [a, d], [b, d], [c, d], [a, e], [b, e], [c, e], [d, e], 
[a, f], [a, g],
f^-2 * h = 1, (b^-1 * h)^2 = 1, (c^-1 * h)^2 = 1, 
(d^-1 * h)^2 = 1,
(e^-1 * h)^2 = 1, g^-1 * h * g * h = 1, 
a * b^-1 * a^-1 * b * d^-1 = 1,
a * c^-1 * a^-1 * c * e^-1 = 1, f^-1 * d^-2 * f * d = 1,
g^-1 * e^-1 * d^2 * g * d^-1 = 1, g * e * d^2 * g^-1 * d^-1 = 1,
e * f^-1 * e^-1 * d * f * e = 1, (g^-1 * b^-1)^3 = 1, 
(g^-1 * c^-1)^3 = 1,
g * c^-1 * a^-1 * b^-1 * a * g^-1 * b^-1 * c = 1,
f^-1 * b^-1 * a * b^-1 * a^-1 * f * b * e^-1 = 1,
a * c^-1 * f^-1 * b^-1 * a^-1 * c * f * c^-1 = 1,
g^-1 * f^-1 * g^-1 * f^-1 * g^-1 * f * g^-1 * f * g^-1 * f = 1 >
\end{verbatim}

\subsection{Degrees $49$ and $121$}

While we were not able to compute the Isaacs groups of degree $49$ and $121$, we were able to find the $p$-groups of order $7^6$ and $11^6$ whose centers have order $49 = 7^2$ or $121 = 11^2$ and whose irreducible characters contain either $48$ characters of degree $49$ or $120$ characters of degree $121$, and whose automorphism groups have orders divisible by $48$ or $120$.  Again, in both cases, we obtain the Heisenberg group of order $p^6$, and two other semi-extraspecial groups of exponent $p^2$ that are isoclinic to the Heisenberg group.  We also obtain one group that has nilpotence class $3$ and exponent $p$.  So as in the case with $p = 5$, for $p = 7$ and $p = 11$, we obtain four possible groups for the normal Sylow $p$-subgroup for a $p$-closed Isaacs group of order $p^6$.  We then ran into difficulties trying to pin down the subgroups of the autormorphism groups to build the complements.

For non closed version, we were able to find the $p$-groups of order $7^5$ that meet our hypotheses, and assuming our calculations in Magma our correct, no non $7$-closed examples exist of Isaacs group of degree $7^2$.  On the other hand, we were not able to complete these calculations for the non $11$-closed Isaacs groups of degree $11^2$.


\end{document}